\documentclass[12pt,a4paper,oneside]{amsart}
\usepackage{amsmath,amssymb,amsthm,hyperref}
\usepackage{enumerate}
\usepackage{amsfonts}
\usepackage{cancel}
\usepackage{t1enc}
\usepackage{cleveref}
\usepackage{graphicx}
\usepackage{mathrsfs}
\usepackage{mathptmx}
\usepackage{lipsum}
\usepackage{tikz}
\usepackage{version} 
\usepackage{color}
\usepackage{datetime}
\numberwithin{equation}{section}
\usepackage[bf]{caption}     
\usepackage{geometry}        
\usepackage{bm}  
\usepackage{buczosierlyukas}
\usepackage{mathtools}
\usepackage[utf8]{inputenc}
\usepackage[normalem]{ulem}

\newtheorem{theorem}{Theorem}[section]

\newtheorem{lemma}[theorem]{Lemma}

\newtheorem{corollary}[theorem]{Corollary}
\DeclareMathAlphabet{\mathpzc}{OT1}{pzc}{m}{it}
{\theoremstyle{definition}
\newtheorem{definition}[theorem]{Definition}
\newtheorem{remark}[theorem]{Remark}

\newtheorem{question}[theorem]{Question}

}

\newcommand{\Id}{\mathrm{Id}}

\newcommand{\dub}{D_{\overline{B}*}}

\definecolor{aquam}{rgb}{0.5,1.0,1.0}
\definecolor{bbrown}{rgb}{0.75,0.38,0.15}
\definecolor{Cyan}{rgb}{0,0.6,0.6}
\definecolor{Darkblue}{rgb}{0,0,1}
\definecolor{Dodgerblue2}{rgb}{0,0.5,1}
\definecolor{Green}{rgb}{0,0.6,0.06}
\definecolor{Kahki}{rgb}{1,1,0.5}
\definecolor{Magenta}{rgb}{0.7,0,0.7}
\definecolor{bMagenta}{rgb}{1,.6,1}
\definecolor{Orange}{rgb}{0.8,0.3,0}
\definecolor{dOrchid}{rgb}{0.7,0.2,0.4}
\definecolor{Orchid}{rgb}{1,0.5,1}
\definecolor{Purple}{rgb}{0.65,0.07,0.85}
\definecolor{Royalblue}{rgb}{0.6,0.85,0.87}
\definecolor{Tan}{rgb}{0.54,0.42,0.23}
\definecolor{bTan}{rgb}{0.94,0.82,0.63}
\definecolor{zoltan}{rgb}{0,0.1,0.3}
\definecolor{Turquoise}{rgb}{0,0.85,0.87}
\definecolor{Yellow}{rgb}{1,1,0}
\definecolor{darkamber}{rgb}{0.4,0.19,0.28}
\definecolor{bYellow}{rgb}{1,1,0.6}
\definecolor{bRed}{rgb}{1,0.7,0.7}
\definecolor{boxcolb}{rgb}{0.87,0.77,0.75}
\definecolor{boxcol}{rgb}{0.6,0.85,0.87}
\definecolor{boxcolgreen}{rgb}{0.64,0.93,0.79}
\definecolor{boxcolaa}{rgb}{.75,.99,.70}
\definecolor{boxcolbb}{rgb}{0.39,0.50,0.56}
\definecolor{boxcolcc}{rgb}{1,0.81,0.65}
\definecolor{yy}{rgb}{0.43,0.21,.18}
\definecolor{gA}{gray}{0.5}
\definecolor{gB}{gray}{0.8}
\definecolor{gC}{gray}{0.9}

\usepackage{xcolor}

\hypersetup{
  colorlinks   = true, 
  urlcolor     = blue, 
  linkcolor    = blue, 
  citecolor   = red 
}

\begin{document}

\title{Exact formula on upper box dimension of generic H\"older level sets}
\author{Zolt\'an Buczolich$^*$}
\address[Z. Buczolich]{Department of Analysis, ELTE E\"otv\"os Lor\'and University, 
P\'azm\'any P\'eter S\'et\'any 1/c, 1117 Budapest, Hungary}
\address[B. Maga]{Alfr\'ed R\'enyi Institute of Mathematics, Re\'altanoda street 13-15, 1053 Budapest, Hungary}
\email{zoltan.buczolich@ttk.elte.hu}
\urladdr{http://buczo.web.elte.hu, ORCID Id: 0000-0001-5481-8797}

\author{Bal\'azs Maga$^\text{\textdagger}$}
\email{mbalazs0701@gmail.com}
\urladdr{  http://magab.web.elte.hu/}


\thanks{\scriptsize $^\text{\textdagger}$ This author was supported by the 
Hungarian National Research, Development and Innovation Office-NKFIH, Grant no. 152535 and no. 152922.
}

\date{\today}

\begin{abstract}
    In the previous decades, the size of level sets of functions have been extensively studied in various setups involving different regularity properties and
    size notions. In the case of Hölder functions, the authors have provided various bounds, but to date no explicit formulae have been found for any studied dimension 
    and the results were valid only about very specific fractals.
    In this paper, for the first time, we have a result valid for a large class of self-similar sets, namely we prove that for these 
    fractals Lebesgue almost every level set of the generic 1-Hölder-$\alpha$ function defined on $F\subseteq \mathbb{R}^p$ has 
    upper box dimension $\dim_H F - \alpha$. 
\end{abstract}

\maketitle

\setcounter{tocdepth}{3}

\tableofcontents


\section{Introduction}

Level sets of continuous functions defined on some domain $F\subseteq \mathbb{R}^p$ convey information about the dimensionality of $F$. Roughly speaking, a generic continuous function should have large level sets
if $F$ itself is thick in an intuitive sense, and small otherwise. This heuristic prompted the investigation of the size of such level sets in terms of Hausdorff dimension
(see \cite{BBElevel, BBEtoph, BK}), with the quantity providing the answer (the topological Hausdorff dimension) attracting interest in its own right \cite{MaZha}, even in physics
\cite{Balankintoph, Balankinfracspace, Balankintransport, Balankintoph2, Balankinfluid}.

The authors with Gáspár Vértesy initiated the study of the Hausdorff and box dimension of level sets of generic 1-Hölder-$\alpha$ functions \cite{ BuczolichMagaVertesy2025, sierc, sier, BUCZOLICH2024531}. 
Under self-similarity conditions on the domain $F$, they proved that for every $\alpha$ and for either of the Hausdorff/lower box/upper box dimension, there is a constant which is the dimension
of Lebesgue almost every level set of the generic 1-Hölder-$\alpha$ function. For any underlying dimension, these values form an interesting spectrum describing the geometry of $F$. 
However, dealing with this spectrum is far from trivial in any of the cases: the existing literature focuses on particular choices of $F$ and provides mostly estimates at the price of
a significant amount of technicalities, highly depending on the structure of $F$. In this direction, in \cite{BuczolichMagaVertesy2025}, 
we managed to provide lower and upper bounds on the Hausdorff spectrum for the Sierpiński triangle which asymptotically coincide as $\alpha\to 0+$.
However, no explicit formulae or general machinery have been found before for describing any of the spectra for a larger class of fractals. In this respect this paper is the first one. 

\subsection{The main result}

We give an explicit formula for the upper box dimension $\overline{\dim}_B$ of Lebesgue almost every level set of a generic 1-Hölder-$\alpha$ function 
valid for connected self-similar sets subject to the geometric condition of having {\it finitely many directions}, i.e.,
whose defining similarities have orthogonal parts generating a finite subgroup of the orthogonal group.

\begin{theorem} \label{thm:gen_upper_box_precise}
    Assume that $F\subseteq \mathbb{R}^p$ is a connected self-similar set with finitely many directions
    and $\dim_H F = s > 1$.
    Then there exists a dense $G_\delta$ set $\mathcal{G}$ of 1-Hölder-$\alpha$ functions from $F$ to $\mathbb{R}$ such
    that for every $f\in \mathcal{G}$ and Lebesgue almost every $r\in f(F)$,
    \begin{displaymath}
    \overline{\dim}_B f^{-1}(r)=s-\alpha.
    \end{displaymath}
\end{theorem}

\subsection{Organization of the paper}

In Section \ref{sec:prelim} we introduce some notation and recall some relevant tools. 

In Section \ref{sec:box_dimension_locally_connected} we provide a lemma on how a dense set of functions over which the upper box dimension of level sets is controlled can be used
to get bounds over a dense $G_\delta$ set of functions. Such a lemma was previously provided by \cite{BUCZOLICH2024531} assuming a more elaborate connectivity property,
this time it will be sufficient to assume local connectivity.

In Section \ref{sec:lemmas} we prove two key lemmas, which do the heavy lifting for Section \ref{sec:main}, in which we prove Theorem \ref{thm:gen_upper_box_precise}. 
We note that in \cite[Theorem~3.6]{BUCZOLICH2024531}, we proved a lower bound for the upper box dimension of level sets under certain conditions using Theorem \ref{prop:wenxi} (identical to
\cite[Proposition~2.6]{BUCZOLICH2024531}). A subtle point in applying this statement was not addressed explicitly there.
Lemma \ref{lemma:values_admitted_in_simplices} fills this important geometric gap, essentially stating
that if we have a piecewise affine function, defined on simplices containing our fractal $F$ then with an arbitrarily small translation
we can make sure that the function takes Lebesgue almost all of its values in the interior of the simplices. This fact is quite easy if the Hausdorff dimension of $F$ is less than $2$,
for higher dimensions it needs some extra care.

We conclude the paper with some open problems.

\section{Preliminaries} \label{sec:prelim}

\subsection{Notation}

For $F\subseteq \mathbb{R}^p$, $C_c^\alpha(F)$ denotes the set of $c$-Hölder-$\alpha$ functions, i.e., the functions satisfying $|f(x)-f(y)|\leq c|x-y|^{\alpha}$ for all $x, y\in F$. We also use the notation
$C_{c-}^{\alpha}(F)=\bigcup_{c'<c}C_{c'}^\alpha(F)$. The term \emph{almost every} is always understood in Lebesgue sense.

Given a function $f:F\to \mathbb{R}$, we put $D_{\overline{B}}^f(r,F)=\overline{\dim}_B f^{-1}(r)$. For $f:F\to \mathbb{R}$, we denote by $D_{\overline{B}*}^f(F)$  the essential infimum
of the upper box dimension of non-empty level sets, i.e., 
\begin{displaymath}
    D_{\overline{B}*}^f(F)= \begin{cases}
        \inf\{d: \lambda\{r : r\in f(F) \text{ and } D_{\overline{B}}^f(r,F)\leq{d}\}>0\}, & \text{if } \lambda(f(F))>0, \\
        0, & \text{if } \lambda(f(F))=0.
    \end{cases}
\end{displaymath}
We denote by $\mg_{1,\aaa}(F)$, or by simply $\mg_{1,\aaa}$ the  set  of dense $G_{\ddd}$
sets in $C_{1}^{\aaa}(F)$. Whether $D_{\overline{B}*}^f(F)$ has a generic value is a priori unclear, nevertheless the following quantity is always well-defined:
\begin{equation}\label{def:generic_upper}
    D_{\overline{B}*}(\aaa,F)=\inf_{\cag\in \mg_{1,\aaa}}\sup\{ D_{\overline{B}*}^{f}(F):f\in \cag \}.
 \end{equation}
As established by \cite[Theorem~3.8]{BUCZOLICH2024531}, subject to certain conditions on the connectivity properties of $F$, $D_{\overline{B}*}(\aaa,F)$ is indeed simply the generic value of $D_{\overline{B}^*}^f(F)$.
In Section \ref{sec:box_dimension_locally_connected} we cite these results and discuss how they extend to locally connected $F$. 
We note that even if $D_{\overline{B}*}(\aaa,F)$ is the generic value of $D_{\overline{B}^*}^f(F)$, it does not necessarily coincide with $D_{\overline{B}}^f(r,F)$ for almost every $r\in f(F)$.
In particular it is not clear a priori why in the setup of Theorem \ref{thm:gen_upper_box_precise} almost every level set of the generic 1-Hölder-$\alpha$ should have the same upper box dimension.

\subsection{Earlier results concerning box dimension of level sets}  \label{subsec:defs}

The following lemma is provided by \cite{BUCZOLICH2024531}, quickly giving rise to an upper bound on the upper box dimension of level sets.

\begin{lemma}[{\cite[Lemma~3.9]{BUCZOLICH2024531}}] \label{lemma:from_covering_to_box}
    Assume that $0<\alpha\leq 1$, and the measurable set $F\subseteq \mathbb{R}^p$ has coverings $(\mathcal{S}_n)_{n=1}^{\infty}$, 
    such that with some constants $C,l,\rho>1$:
    \begin{itemize}
        \item the cardinality of $\mathcal{S}_n$ is at most $Cl^n$ for some $C,l>1$,
        \item if $S\in \mathcal{S}_n$, then $\diam(S)\leq C\rho^{-n}$.
    \end{itemize}
    Then for any $f\in C_{\alpha}^{1}(F)$ and almost every $r\in \mathbb{R}$, we have
    $\overline{\dim}_B(f^{-1}(r))\leq \frac{\log l}{\log \rho} - \alpha$. In particular,
    $D_{\overline{B}*}(\alpha, F)\leq \frac{\log l}{\log \rho} - \alpha$.
\end{lemma}

Lemma \ref{lemma:from_covering_to_box} has the following immediate corollary, only noted in the special case of $F$ being the Sierpi\'nski triangle in \cite{BUCZOLICH2024531}:

\begin{corollary} \label{corollary:lower_box_dim_upper_bound}
    Assume that $0<\alpha\leq 1$, and let $F\subseteq \mathbb{R}^p$ be measurable. Then for every $f\in C_{\alpha}^{1}(F)$ and almost every $r\in \mathbb{R}$, we have
    $\overline{\dim}_B(f^{-1}(r))\leq \underline{\dim}_B(F) - \alpha$. In particular,
    $D_{\overline{B}*}(\alpha, F)\leq \underline{\dim}_B(F) - \alpha$.
\end{corollary}

Recalling \cite{falconer1989dimensions}, we know that the box dimension and the Hausdorff dimension of a self-similar set coincide, thus
Corollary \ref{corollary:lower_box_dim_upper_bound} immediately implies one of the bounds constituting our main result.

We recall a further tool from \cite{sier} ,which we slightly refine for our purposes right away. 
Given a locally finite set $\mathcal{S}$ of non-overlapping, non-degenerate simplices, we say that $f:\bigcup \mathcal{S}\to \mathbb{R}$ is piecewise affine
if it is affine in each $S\in \mathcal{S}$. The following statement is formally stronger then what is given in \cite{sier}, the proof actually guarantees
this:

\begin{lemma}[{\cite[Lemma~4.4]{sier}}] \label{lemma:piecewise_affine_approx}
Assume that $F$ is compact and $c > 0$ is fixed.
Consider the locally non-constant, piecewise affine $c^{-}$-Hölder-$\alpha$ functions with domain $\mathbb{R}^p$.
Restricting these to $F$ we get a dense subset of $C_{c}^{\alpha}(F)$.
\end{lemma}

\begin{remark}
Note that distinct $f\neq g$ might still result in the same $f|_{F}=g|_{F}$.
This is trivial of course, the reason why we highlight this is that $f$ and $g$ might have different moduli of continuity. It is important to emphasize
that the $c^{-}$-Hölder-$\alpha$ functions guaranteed by \Cref{lemma:piecewise_affine_approx} are $c^{-}$-Hölder-$\alpha$ even on the full space.
\end{remark}

\subsection{Various further tools}

The following extension theorem is a fundamental tool in our arguments:

\begin{theorem}[{\cite[Theorem~1]{[GrunbHolderext]}}] \label{thm:Grunb}
    Suppose that $F\sse \R^{p}$ and $f:F\to \R$ is a $c$-H\"older-$\alpha$ function.
    Then there exists a $c$-H\"older-$\alpha$ function $g:\R^{p}\to \R$
     such that
     $g(x)=f(x)$ for $x\in F$.
\end{theorem}

We will have to understand the level sets of piecewise affine functions. These level sets consist of partial hyperplanar slices of the fractal $F$. The dimension of such slices is the subject
of the vast theory of slicing theorems. On an intuitive level, the statement is that in almost every direction a positive measure of slices has dimension $\dim F -1$, where 
$\dim$ is either Hausdorff, box, or packing dimension, and assuming some homogeneity almost every slice intersecting $F$ has this dimension. While the following theorem
is originally stated for slices of any dimension, for technical simplicity we only state it for $p-1$, i.e., hyperplanar slices. 
We say that a subset $\mathcal{W}$ of $(p-1)$-dimensional linear 
subspaces is of full measure if the set of corresponding unit normal vectors $\{\pm n_{W}:W\in \mathcal{W}\}$ 
is a full measure set of the unit sphere $\mathbb{S}^{p-1}\subseteq \mathbb{R}^p$. Below $\dim$ is either the Hausdorff,
the lower box, the upper box, or the packing dimension.

\begin{theorem}[{\cite[Corollary~1]{wenxi}}] \label{prop:wenxi}
    Assume that $A\subseteq \mathbb{R}^p$ is a self-similar set with finitely many directions and $\dim_H A>1$.
    Then for a full measure set $\mathcal{W}_A$ of the $(p-1)$-dimensional linear subspaces, if $W\in \mathcal{W}_A$, then 
    $$\mathcal{H}^1(\mathrm{Pr}_{W^{\perp}}(A)\setminus \{a\in W^{\perp}: \dim (A\cap (W+a))= \dim_H A  -1\}) = 0,$$
    that is standard dimension drop occurs on a set of full measure. (Here $W^{\perp}$ denotes the 1-dimensional orthocomplement of $W$, 
    and $\mathrm{Pr}_{W^{\perp}}$ is the projection onto it.)
\end{theorem}

We will use the terminology that an affine function is \emph{typically oriented for $A$} (or simply \emph{typically oriented}, if it does not cause ambiguity)
if its level sets are parallel to a linear subspace in $\mathcal{W}_A$ under the notation of Theorem \ref{prop:wenxi}.
(Or equivalently, its gradient is the normal vector of an element of $\mathcal{W}_A$.) 

Finally, we will need an understanding of the topology of self-similar sets. It is not difficult to see that if a set is the finite union of connected sets
of diameter at most $\varepsilon$ for any $\varepsilon$, then it is locally connected (see e.g. \cite[(15.1)]{Whyburn_1942}). 
Considering the similar images of a connected self-similar set, we find that it obviously has this property, implying the following:

\begin{corollary} \label{cor:connected_self_similar_is_locally_connected}
    Any connected self-similar set is locally connected.
\end{corollary}

We note that this result is significantly improved by \cite{Hata1985}, turning out to be valid for attractors of weakly contractive IFS's.

\subsection{Auxiliary functions} \label{subsec:auxiliary_functions}

We will use a two-parameter family of auxiliary functions below, our aim is to guarantee level sets with large dimension. A similar family of functions was already introduced in \cite{BUCZOLICH2024531}.

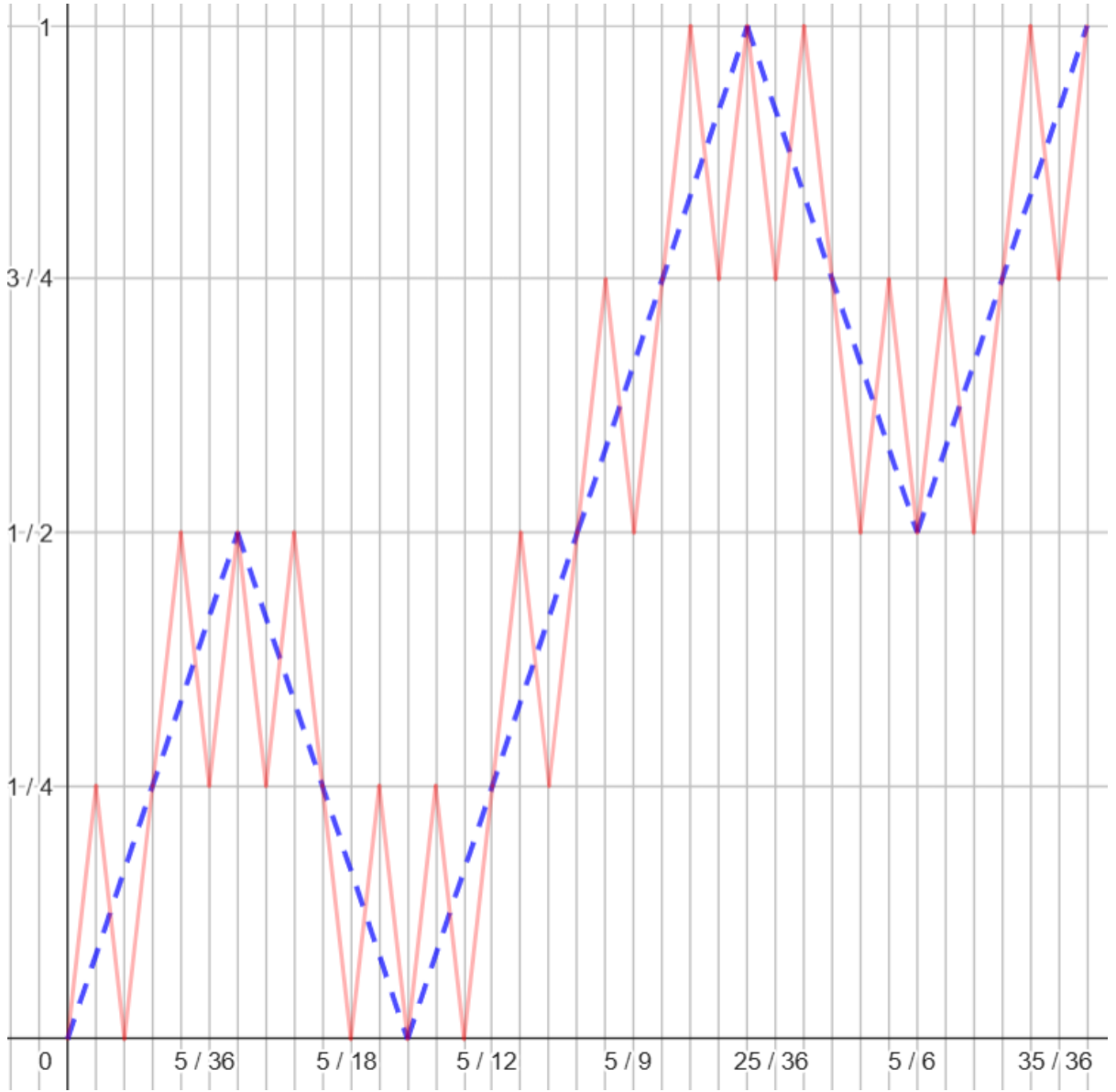
\begin{figure}[h]
    \centering
    \begin{tikzpicture}[
    x=13cm,
    y=13cm,
    every node/.style={font=\large}
]

\begin{scope}
    \clip (-0.035,-0.04) rectangle (1.015,1.015);

    \foreach \k in {0,...,36}{
        \draw[gray!30, line width=0.4pt]
            ({\k/36},0) -- ({\k/36},1);
    }

    \foreach \k in {1,...,6}{
        \draw[gray!55, line width=0.8pt]
            ({\k/6},0) -- ({\k/6},1);
    }

    \foreach \k in {0,...,4}{
        \draw[gray!40, line width=0.45pt]
            (0,{\k/4}) -- (1,{\k/4});
    }

    \draw[
        red!40,
        line width=1.25pt,
        line join=round
    ]
        (0,0)
        -- ({1/36},{1/4})
        -- ({2/36},0)
        -- ({3/36},{1/4})
        -- ({4/36},{1/2})
        -- ({5/36},{1/4})
        -- ({6/36},{1/2})
        -- ({7/36},{1/4})
        -- ({8/36},{1/2})
        -- ({9/36},{1/4})
        -- ({10/36},0)
        -- ({11/36},{1/4})
        -- ({12/36},0)
        -- ({13/36},{1/4})
        -- ({14/36},0)
        -- ({15/36},{1/4})
        -- ({16/36},{1/2})
        -- ({17/36},{1/4})
        -- ({18/36},{1/2})
        -- ({19/36},{3/4})
        -- ({20/36},{1/2})
        -- ({21/36},{3/4})
        -- ({22/36},1)
        -- ({23/36},{3/4})
        -- ({24/36},1)
        -- ({25/36},{3/4})
        -- ({26/36},1)
        -- ({27/36},{3/4})
        -- ({28/36},{1/2})
        -- ({29/36},{3/4})
        -- ({30/36},{1/2})
        -- ({31/36},{3/4})
        -- ({32/36},{1/2})
        -- ({33/36},{3/4})
        -- ({34/36},1)
        -- ({35/36},{3/4})
        -- (1,1);

    \tikzset{
        blue graph/.style={
            blue!70,
            line width=2.1pt,
            dash pattern=on 6pt off 5pt,
            line cap=butt
        }
    }

    \draw[blue graph]
        (0,0) -- ({1/6},{1/2});

    \draw[blue graph]
        ({1/3},0) -- ({1/6},{1/2});

    \draw[blue graph]
        ({1/3},0) -- ({2/3},1);

    \draw[blue graph]
        ({5/6},{1/2}) -- ({2/3},1);

    \draw[blue graph]
        ({5/6},{1/2}) -- (1,1);
\end{scope}

\draw[black!65, line width=0.7pt]
    (0,-0.045) -- (0,1.015);

\draw[black!65, line width=0.7pt]
    (-0.025,0) -- (1.015,0);

\node[below left=2pt, fill=white, inner sep=1pt]
    at (0,0) {$0$};

\foreach \k in {1,...,5}{
    \node[below=3pt, fill=white, inner sep=1pt]
        at ({\k/6},0) {$\frac{\k}{6}$};
}

\node[below=3pt, fill=white, inner sep=1pt]
    at (1,0) {$1$};

\node[left=4pt, fill=white, inner sep=1pt]
    at (0,{1/4}) {$\frac14$};

\node[left=4pt, fill=white, inner sep=1pt]
    at (0,{1/2}) {$\frac12$};

\node[left=4pt, fill=white, inner sep=1pt]
    at (0,{3/4}) {$\frac34$};

\node[left=4pt, fill=white, inner sep=1pt]
    at (0,1) {$1$};

\end{tikzpicture}
    \caption{The first two iterations of constructing the auxiliary function $\phi_{3, 2}$ over $[0, 1]$.}
    \label{fig:auxiliary_function}
\end{figure}

The function $\phi_{n, m}$ is defined by recursively given approximations, see Figure \ref{fig:auxiliary_function}. For parameters $m\geq 2$ and $n>1$ odd, we first define a piecewise linear function 
$\phi_{n, m}^{(1)}$ on $[0, 1]$ as follows: we consider the $\frac{1}{nm}$ grid intervals in $[0, 1]$, $\phi_{n, m}^{(1)}$ 
is going to be linear on each of these. The values in the endpoints
are defined such that $\phi_{n, m}^{(1)}$ oscillates between $\frac{i}{m}$ and $\frac{i+1}{m}$ on the interval $[\frac{i}{m}, \frac{i+1}{m}]$. 
As $n$ is odd, this defines a continuous function.

We obtain $\phi_{n, m}^{(l+1)}$ by replacing each linear piece of the graph of $\phi_{n, m}^{(l)}$ by an affine image of the graph of $\phi_{n, m}^{(1)}$, 
such that the values in the endpoints above are unchanged. This iteration tends to a continuous function $\phi_{n, m}$, which is 1-Hölder-$\alpha_{n,m}$ for 
$\alpha_{n,m} = \frac{\log m}{\log n  + \log m}$. Using the mass distribution principle, it is easy to show that every non-empty level set of $\phi_{n, m}$ has Hausdorff dimension $(1-\alpha_{n,m})$.
(They even have positive $(1-\alpha_{n,m})$-dimensional measure, similarly to the construction in \cite{Rajala2026}.) 
Possible $\alpha_{n, m}$s form a dense set in $[0, 1]$. This observation will have large importance.

Finally, extend $\phi_{n, m}$ to $\mathbb{R}$ with $\phi_{n, m}(x+k)=\phi_{n, m}(x)+k$ for every $x\in [0, 1)$, $k\in \mathbb{Z}$. Note that the extended
$\phi_{n, m}$ is not 1-Hölder-$\alpha_{n, m}$ globally, but it is 1-Hölder-$\alpha_{n, m}$ in every interval of length 1.  
We remark that $\phi_{n, m}$ maps any grid interval of length $\frac{1}{m}$ to itself, consequently, 
$\|\phi_{n, m} - \Id_{\mathbb{R}}\|_{\infty}\leq \frac{1}{m}$.

Consider the intervals $\left[\frac{i}{m^k}, \frac{i+1}{m^{k}}\right]$ for $i\in\mathbb{Z}$, denote their family by $\mathcal{I}_k$. 
Fix $I\in \mathcal{I}_k$. Observe that by induction, $n^k$ different grid intervals of length $(nm)^{-k}$ are mapped bijectively to $I$ by a maximal monotone piece of $\phi_{n, m}^{(k)}$.
Denote their family by $\mathcal{I}_{I, k}$.
By construction, the $\phi_{n, m}^{(k)}$-image of any $I'\in \mathcal{I}_{I, k}$ is the same as its $\phi_{n, m}$-image, thus $\phi_{n, m}^{-1}(I)=\bigcup \mathcal{I}_{I, k}$.
\begin{lemma} \label{lemma:measure_multiplying}
    $\phi_{n, m}$ is a measure-preserving function. More generally, if 
    $I'\in \mathcal{I}_{I, k}$, it is mapped to $I$ in a measure multiplying way, i.e., for every measurable $V\subseteq I$
    we have $\lambda(V)= n^k\lambda(\phi_{n, m}^{-1}(V)\cap I')$.
\end{lemma}

\begin{proof}
    The first claim follows from the second indeed, applied for $k=0$: that shows that $\phi_{n, m}$ is measure-preserving restricted to any invariant set of the form $[i, i+1]$,
    and consequently it is measure-preserving over the full line.
    
    For the second claim, take any grid subinterval $J$ of $I$ in $\mathcal{I}_l$, $l\geq k$. Restricting the the discussion leading up to the lemma to $I'$, we see
    that $\mathcal{I}_{J, l}$ contains $n^{l-k}$ intervals of length $(nm)^{-l}$, and $\phi_{n, m}^{-1}(J)\cap I'$ is simply the union of these. This proves the claim
    of the lemma for $V=J$. As any relative open set $V\subseteq I$ can be written as the countable union of elements of $\bigcup_{l=k}^{\infty}\mathcal{I}_l$,
    the claim of the lemma immediately extends to any relative open $V$, and then by taking complements to any closed subset $V$ of $I$. By the regularity of $\lambda$, it extends
    then to all measurable $V\subseteq I$.
\end{proof}

We observe $\phi_{n, m}(1-x)=1-\phi_{n, m}(x)$, and take note of the following self-affine property of the graph of $\phi_{n, m}$: 
over any interval of the form
$\left[\frac{in+k}{nm}, \frac{in+k+1}{nm}\right]$, $i\in\mathbb{Z}$, $0\leq k<n$, if $k$ is even, we have
$$\phi_{n, m}(x)=\frac{i}{m}+\frac{\phi_{n, m}\left(nm\left(x-\frac{in+k}{nm}\right)\right)}{m},$$
while if $k$ is odd, we have
$$\phi_{n, m}(x)=\frac{i+1}{m}-\frac{\phi_{n, m}\left(nm\left(x-\frac{in+k}{nm}\right)\right)}{m}.$$
By iteration, this idea extends to any grid interval $J$ of length $(nm)^{-k}$: there exist affine maps $\psi_{J}^*, \psi_{J}:\mathbb{R}\to \mathbb{R}$ such 
that $\psi_{J}(J)=[0, 1]$ and for every $x\in J$, we have
\begin{equation}\label{*konjeq} 
 \phi_{n, m}(x)=(\psi_{J}^*\circ\phi_{n, m}\circ\psi_{J})(x).
\end{equation}

\section{Generic box dimension of level sets on locally connected domains} \label{sec:box_dimension_locally_connected}

In \cite{BUCZOLICH2024531}, the existence of a dense $G_\delta$ set over which $D_{\overline{B}*}^f(F) = D_{\overline{B}*}(\alpha,F)$ holds is established for fractals
having the following connectivity property defined in terms of certain partitions of $F$:

\begin{definition} \label{def:nicely_connected}
    We say that $(\mathcal{T}_n)_{n=1}^{\infty}$ is a  {\it box dimension defining sequence} (BDDS) of $F\subseteq \mathbb{R}^p$ 
    if 
    \begin{itemize}
        \item  each $\mathcal{T}_n$ is a countable family of closed subsets covering  $F$ 
        in a non-overlapping manner,
        \item $\mathcal{T}_{n+1}$ is a refinement of $\mathcal{T}_{n}$ for each $n\geq 1$, that is for any $T\in\mathcal{T}_n$, 
        there is some $m>0$ and $T_i\in \mathcal{T}_{n+1}$ ($i=1, 2, ..., m$) such that
        $T\cap F=\bigcup_{i=1}^{m}T_i \cap F$,
        \item there exists $C>0$ and $0<q<1$ such that for any $n$ and $H\in \mathcal{T}_n$,
        $H$ contains a ball with diameter $C^{-1}q^{n}$ and $\diam(H) \leq Cq^{n}$. 
         The quantity  $q$ is the {\it base} of $(\mathcal{T}_n)_{n=1}^{\infty}$. 
    \end{itemize}
    
    The set $F\subseteq \mathbb{R}^p$ has {\it nice connection type}, if it has BDDS  
    such that $F\cap H$ has countably many components for any $H\in \bigcup_{n=1}^{\infty}\mathcal{T}_n$.
\end{definition}

As discussed in \cite{BUCZOLICH2024531}, (upper and lower) box dimension can be computed using box dimension defining sequences, and having nice connection type mitigates the handling of
level sets, culminating in the following two statements:

\begin{lemma}[{\cite[Lemma~3.7]{BUCZOLICH2024531}}]\label{lemma:upper_box_generic}
    Suppose that $0< \aaa\leq 1$,  $F\subset\R^p$ is compact with nice connection type, witnessed by $(\mathcal{T}_n)_{n=1}^{\infty}$.
    Moreover, assume that $E=F$ or $E = F \cap H$ for some $H\in \bigcup_{n=1}^{\infty}\mathcal{T}_n$, and $\cau\subset C_1^\alpha(F)$ is open.
    If $\{f_1,f_2,\ldots\}$ is a countable dense subset of $\cau$, then there is a dense $G_\delta$ subset $\cag$ of $\cau$ such that 
    \begin{equation}\label{eqdfb_upper_boxdim}
    \inf_{f\in\cag} D_{\overline{B}*}^f(E) \ge \inf_{k\in\N} D_{\overline{B}*}^{f_k}(E).
    \end{equation}
\end{lemma}

\begin{theorem}[{\cite[Theorem~3.8]{BUCZOLICH2024531}}] \label{thm:upper_box_generic}
    Suppose that $0< \aaa\leq 1$,  $F\subset\R^p$ is compact with nice connection type, witnessed by $(\mathcal{T}_n)_{n=1}^{\infty}$.
    Then there is a dense $G_\delta$ subset $\cag$ of $C_1^\alpha(F)$ 
    such that for every $f\in\cag$ we have $D_{\overline{B}*}^f(F) = D_{\overline{B}*}(\alpha,F)$.   
\end{theorem}

We would like to apply these results on arbitrary connected self-similar sets. By Corollary \ref{cor:connected_self_similar_is_locally_connected}, we know that such sets are locally connected,
however, it seems to be a rather nontrivial geometric question how local connectedness relates to nice connection type. It turns out that instead of understanding
this relationship, it is easier to adjust the proof of
the previous two results to the case of locally connected sets. The following proof is highly similar to the proof of Lemma \ref{lemma:upper_box_generic}. 

\begin{lemma}\label{lemma:upper_box_generic_locally_connected}
    Suppose that $0< \aaa\leq 1$,  $F\subset\R^p$ is a locally connected, bounded, measurable set.
    Moreover, assume that $E = F \cap H$ for some non-empty open set $H$ and $\cau\subset C_1^\alpha(F)$ is open.
    If $\{f_1,f_2,\ldots\}$ is a countable dense subset of $\cau$, then there is a dense $G_\delta$ subset $\cag$ of $\cau$ such that 
    \begin{equation}\label{eqdfb_upper_boxdim_locally_connected}
    \inf_{f\in\cag} D_{\overline{B}*}^f(E) \ge \inf_{k\in\N} D_{\overline{B}*}^{f_k}(E).
    \end{equation}
\end{lemma}

\begin{proof}[Proof of Lemma \ref{lemma:upper_box_generic_locally_connected}]

    The intersection of a locally connected and an open set, $E$ is locally connected as well. By separability, $E$ has countably many components. 
    If some $f_k$ is constant on all of these components,
    then the right hand side of \eqref{eqdfb_upper_boxdim_locally_connected} vanishes by definition and we have nothing to prove.
    Assume that it is not the case, then by continuity $\lambda(f_k(E))>0$ for all $k$.

    Since countable union of sets of measure zero is still of measure zero we can choose
    a measurable set $R_{0}\sse \R$ such that $\lll(\R\sm R_{0})=0$ and  for any $k$ 
    \begin{equation}\label{*lrdfkr}
    \dub^{f_{k}}(r,E)\ge  \inf_{k'\in\N} \dub^{f_{k'}}(E) \text{ for any $r\in R_{0}\cap f_{k}(E)$.}
    \end{equation}
    Fix an enumeration $\mathcal{V} = \{V_1, V_2, \dots\}$ of balls with rational center and radius. 
    By our assumption on local connectedness and separability, $E_j = E\cap V_j$ consists of countably many components for every $j$, which we denote by $E_j^1, E_j^2, \dots$. Thus we can also assume that
    $R_0$ does not contain $\min_{E_j^i}f_k$ and $\max_{E_j^i}f_k$ for any $k$ and component $E_{j}^i$ of the above form. (If these extreme values fail to exist for some indices $i, j, k$, then
    these indices simply do not impose restrictions on $R_0$.)

    The family $\mathcal{V}$ can be used to compute the box dimension of subsets of $E$ as follows: 
    we can create locally finite subfamilies $\mathcal{V}_n$ for $n=1, 2, \dots$ such that the balls in $\mathcal{V}_n$ still cover $\mathbb{R}^p$,
    all of them have radius $2^{-n}$ and any ball $V\in \mathcal{V}_n$ intersects at most $K$ others in $\mathcal{V}_n$ for some constant $K$ depending only on the dimension, $p$. Then if we denote 
    by $A_n(H)$ the set of elements of $\mathcal{V}_n$ intersected by a bounded set $H\subseteq \mathbb{R}^p$ and by $a_n(H)$ the cardinality of $A_n(H)$,
    $$\liminf \frac{\log a_n(H)}{n\log 2} = \underline{\dim}_B(H), \quad \limsup \frac{\log a_n(H)}{n\log 2} = \overline{\dim}_B(H).$$

Suppose $k\in \N$ is fixed, moreover $r\in R_0\cap f_{k}(E)$ and $m\in\N $ are given, we will think of them as variables and on $\N$ we will use the discrete topology.


There is a radius $\rrr_{k,m,r}>0$ such that for any $g \in B(f_k, \rrr_{k, m, r})$, we have that for $V_j\in \mathcal{V}_{m}$, the containment
    $V_j\in A_{m}(f_k^{-1}(r)\cap E)$ implies $V_j\in A_{m}(g^{-1}(r)\cap E)$. 
    Indeed, if $V_j\in A_{m}(f_k^{-1}(r)\cap E)$, we can find a component $E_j^i$ such that $r \in \inte f_k(E_j^i)$, thus there exists
    a radius $\rrr_{k,m,r,j}>0$ such that $V_j\in A_{m}(g^{-1}(r)\cap E)$ for $g\in B(f_k, \rrr_{k, m, r,j})$.
    As the set 
    $A_{m}(f_k^{-1}(r)\cap E)$ is finite, 
    $$\rrr_{k, m, r} = \min_{j:V_j\in A_{m}(f_k^{-1}(r)\cap E)} \rrr_{k,m,r,j}$$
    is a suitable choice. Hence $a_{m}(f_k^{-1}(r)\cap E)\leq a_{m}(g^{-1}(r)\cap E)$.

    We prescribe these radii explicitly to assure measurability. That is for fixed $j$, let
    $$\rrr_{k,m,r,j} = \frac{1}{2}\sup_{i :r \in \inte f_k(E_j^i)}\min \{|r- \inf_{x\in E_j^i} f_k(x)|, |r- \sup_{x\in E_j^i} f_k(x)|, 1\},$$
    which,  as the supremum of countably many continuous functions, is a positive measurable function of $r\in R_0 \cap f_k(E)$ and $m\in \N$. 
    The definition of $\rrr_{k, m, r}$ becomes valid with $$\rrr_{k, m, r} = \min_{j:V_j\in \mathcal{V}_{m},\ V_{j}\in A_{m}(f_k^{-1}(r)\cap E)}\rrr_{k,m,r,j},$$ which is
    a positive measurable function as the minimum of a finite set of positive measurable functions.

Observe that
\begin{equation}\label{*meas1} 
a_{m}(f_k^{-1}(r)\cap E)=\sum_{V\in \mathcal{V}_m}\mathbf{1}_{f_{k}(V\cap E)}(r)\end{equation}
is a measurable function of $m$ and $r$: as $E$ is locally connected, $V\cap E$ has countably many components, and hence $f_k(V\cap E)$ is the
union of countably many intervals.

From \eqref{*meas1} it follows that $\log(a_{m}(f_k^{-1}(r)\cap E))/m$ is also a measurable function of $m$ and $r$. 

Suppose that $D_1<\inf_{k\in\N} \dub^{f_k}(E)$ is fixed.

Given $n\in \N$ we put
\begin{equation}\label{*meas2} 
m(n,k,r)=\min  \Big \{ m\geq n:  \frac{\log(a_{m}(f_k^{-1}(r)\cap E))}{m \cdot \log 2} > D_{1} \Big \}.
\end{equation} 
For fixed $k$ it is also a measurable function of $n\in\N$ and $r\in R_0\cap f_{k}(E)$. By our assumptions it is finite on $R_{0}\cap f_{k}(E)$.

Therefore for fixed $k$ and $n$, the function $\rrr(k,m(n,k,r),r)$ is a measurable function of $r$.


    By taking a suitable superlevel set of $\rrr(k,m(n,k,r),r)$, we get a measurable set $R_{k, n}$ such that $\lambda(R_{k,n})>\lambda(f_k(E))-2^{-n}$, and 
    $\rrr_{k, n} :=\inf_{r\in R_{k, n}}\rrr(k,m(n,k,r),r)>0$. As $f_k(E)$ is the union of countably many intervals, by further reducing $\rrr_{k, n}$, we can assume that 
    for any $f \in B(f_k, \rrr_{k, n})$, we have that 
    $$\lambda(f(E) \setminus f_k(E)) < 2^{-n},$$
    and hence
    $$\lambda(f(E) \setminus R_{k,n}) < 2^{-(n-1)}.$$
    Moreover, for any $r\in R_{k,n}$ we have
    $$\frac{\log{a_{m(n,k,r)}(f^{-1}(r)\cap E)
    }}{m(n,k,r) \cdot \log 2}\geq \frac{\log{a_{m(n,k,r)}(f_k^{-1}(r)\cap E)
    }}{m(n,k,r) \cdot \log 2} > D_1.$$
    
    Now let $\cag_{n}=\bigcup_{k}B(f_{k},\rrr_{k,n})\cap \cau$, and
    $\cag =\bigcap_{n}\cag_{n}$.

    Suppose $f\in \cag.$ 
    Then there exists a sequence $k_{n}$ such that $f\in B(f_{k_{n}},\rrr_{k_{n},n})$ for every $n$. Consequently, 
    $$\lambda(f(E) \setminus R_{{k_n}, n}) < 2^{-(n-1)}.$$
    Thus by the Borel--Cantelli lemma, almost every $r\in f(E)$ is contained in infinitely many $R_{{k_n}, n}$. That is, 
    $$\frac{\log{a_{m(n,k,r)}(f^{-1}(r)\cap E)}}{m(n,k,r) \cdot \log 2}\geq D_1$$
    for infinitely many $m(n,k,r)$. Hence $\dub^{f}(r,E)\geq D_1$ for almost every $r\in f(E)$.
    As $D_1<\inf_{k\in\N} \dub^{f_k}(E)$ was chosen arbitrarily, the lemma is proven.
    
\end{proof}

Simply plugging in the definition of $D_{\overline{B}*}^{f_k}(E)$, for $E=F$ and $\cau=C_1^{\alpha}(F)$, we end up with the following well-applicable statement:

\begin{corollary} \label{corollary:upper_box_generic_locally_connected}
    Suppose that $0< \aaa\leq 1$,  $F\subset\R^p$ is a locally connected, bounded, measurable set.
    If $\{f_1,f_2,\ldots\}$ is a countable dense subset of $C_1^\alpha(F)$, and for every $k$ and almost every $r\in f_k(F)$ we have
    $$\overline{\dim}_B f_k^{-1}(r)\geq D,$$
    then there is a dense $G_\delta$ subset $\cag$ of $C_1^{\alpha}(F)$ such that for every $f\in \cag$ and almost every $r\in f(F)$ we have
    $$\overline{\dim}_B f^{-1}(r)\geq D.$$
\end{corollary}

We note that Lemma \ref{lemma:upper_box_generic_locally_connected} implies the following theorem, in a similar manner to how Lemma \ref{lemma:upper_box_generic} implies Theorem \ref{thm:upper_box_generic}.

\begin{theorem} \label{*thm36}  \label{thm:upper_box_generic_locally_connected}
    Suppose that $0< \aaa\leq 1$,  $F\subset\R^p$ is a locally connected, bounded, measurable set.
    Then there is a dense $G_\delta$ subset $\cag$ of $C_1^\alpha(F)$ 
    such that for every $f\in\cag$ we have $D_{\overline{B}*}^f(F) = D_{\overline{B}*}(\alpha,F)$.   
\end{theorem}

As the proof is very similar to deducing Theorem \ref{thm:upper_box_generic} from Lemma \ref{lemma:upper_box_generic} and we will not even directly use Theorem \ref{*thm36} in this paper,
we do not provide the details, we just highlight the differences. These come from Lemmata \ref{lemma:upper_box_generic}-\ref{lemma:upper_box_generic_locally_connected} 
concerning different sets of the form $E=F\cap H$. However,
we only use that such sets of arbitrarily small diameter can be used to cover $F$, which is clearly true in the setup of Lemma \ref{lemma:upper_box_generic_locally_connected} as well.

\section{Approximating affine functions} \label{sec:lemmas}

In light of Corollary \ref{corollary:upper_box_generic_locally_connected}, we will quite quickly reduce Theorem \ref{thm:gen_upper_box_precise} to the question whether
any $C_1^{\alpha}(F)$ function $h$ is arbitrarily well-approximable with $C_1^{\alpha}(F)$ functions which have level sets of dimension about $s-\alpha$. Considering a piecewise affine
approximation first (to be detailed in Section \ref{sec:main}), by self-similarity, a natural building block of this approximability result is that if $h$ is affine, 
such an approximation can be constructed. This approximation will be established by Lemma \ref{lemma:approx_affine_functions}, using the auxiliary functions constructed in Section 
\ref{subsec:auxiliary_functions}.

\begin{lemma} \label{lemma:approx_affine_functions}
    Let $F\sse \R^p$ be a connected self-similar set with lower box dimension $s>1$ and fix a typically oriented affine function $h:\R^p\to\R^p$. 
    
    Then for any $n, m$ we have $\|\phi_{n, m}\circ h -h\| <\frac{1}{m}$ and almost every non-empty level set of $\phi_{n, m}\circ h$ has upper box dimension at least $s-\alpha_{n, m}$.
\end{lemma}

\begin{proof}
    As $\|\phi_{n, m} - \Id_{\mathbb{R}}\|_{\infty}\leq \frac{1}{m}$, the first statement is obvious.

In case $h$ is constant the other statement is obvious. Hence we can suppose that $h$ is non-constant.

  For a nonnegative integer 
    $k$ by a $k$-level grid interval we will always mean some interval $\left[\frac{b}{(nm)^k} , \frac{b+1}{(nm)^k}\right]$ for an integer $b$. 
 
    A level $k$ grid cube is the cartesian product of level $k$-level grid intervals. We will use these to estimate box dimensions, and use the customary 
    notation $N_k(V)$ for the number of level $k$
    grid cubes intersected by a set $V$.

    Put $g_{n, m, h} = \phi_{n, m}\circ h$. Modulo its endpoints, the closed interval $h(F)$
    can be written as a countable union of non-overlapping grid intervals of the form $ \left[\frac{b}{m^k} , \frac{b+1}{m^k}\right]$. 
For any such interval, $J'$ select an arbitrary grid interval $J$ of length $(nm)^{-k}$, such that $\phi_{n,m}(J)=J'$.
    For such an interval $J$,
    let $F_{J} = F\cap h^{-1}(J)$. As $h$ is affine on $F$, $F_{J}$ is a piece of $F$ cut by parallel hyperplanes. 
    Since $F$ is connected none of these planes can cut $h(F)$ into two pieces and hence $h(F_{J})=J$.
    Showing that inside $F_{J}$ almost every level set of $g_{n, m, h}$ has dimension at least $s-\alpha_{n, m}$ is sufficient, 
    as almost every level set intersects
    one of the countably many $F_{J}$s.

    Note that in $F_{J}$, by \eqref{*konjeq}, we have 
    $$g_{n, m, h}=\phi_{n, m}\circ h = \psi_{J}^*\circ\phi_{n, m}\circ \psi_{J}\circ h,$$
    where $\psi_{J}\circ h$ is affine and $(\psi_{J}\circ h)(F_{J})=[0, 1]$. 
    As applying the affine map $\psi_{J}^*$ does not alter the box dimension of almost every non-empty level set,
    it suffices to consider the box dimension of level sets of $\phi_{n, m}\circ (\psi_{J}\circ h)=g_{n, m, \psi_{J}\circ h}$. Thus abusing the notation
    and denoting $\psi_{J}\circ h$ by $h$ instead, without loss of generality, we can assume that already $h(F_{J})=[0, 1]$. For notational simplicity, let us do so. 
    
    Let $W_a$ be the hyperplane for which $h(W_a\cap F_J) = a$. As $h$ is typically oriented, for almost every $a\in h(F_J)=J$ we have
    $\dim_B (W_a \cap F_J) = s -1$ according to the box dimension version of \Cref{prop:wenxi}.

    By the continuity of the Lebesgue measure, for any $\rho>0$ there exists some $k(\rho)$ such that for $k\geq k(\rho)$ 
    there exists some $H=H(k, \rho)\subseteq J$ such that $\lambda(H)>1-\rho$ 
    and for any $a\in H$ we have for the number $N_k(W_a\cap F_J)$ of level $k$ grid cubes intersecting $W_a\cap F_J$ 
    $$N_k(W_a\cap F_J) \geq (nm)^{k(s-1-\rho)}.$$
Note that the measurability of $H(k, \rho)$ 
    follows from the measurability of $a\mapsto N_k(W_a\cap F_J)$: this mapping is actually upper semicontinuous as $F_J$ is closed. In the sequel several sets will depend on $\rho$ and $k$ and at several places for ease of notation we will not make this dependence explicit.

 We will select a $\rho=\rho_{l}\in(0,1/l)$ which is sufficiently small, in fact it satisfies our assumptions after \eqref{*BIveq} and \eqref{eq:no_of_intersected_cubes1}.
 After this we also select and fix a $k_{l}\geq \max\{ l, k(\rho_{l})\}$.
  For ease of notation almost until the end of this proof we will just use the notation $\rho$ and $k$ instead of $\rho_l$ and $k_{l}$, since we will assume almost until the end that $l$ is fixed.

    We recall some notation and observations from Section \ref{subsec:auxiliary_functions}: consider the intervals $\left[\frac{i}{m^k}, \frac{i+1}{m^{k}}\right]$ for $i\in\mathbb{Z}$, denote their family by $\mathcal{I}_k$.
    To any $I\in\mathcal{I}_k$, $n^k$ different intervals of length $(nm)^{-k}$ are mapped by $\phi_{n, m}$. Denote the family of these intervals by $\mathcal{I}_{I, k}$. 
    \medskip

    We claim that for at least $(1-\rho^{1/3})m^k$ choices
    of $I\in \mathcal{I}_k$, there are at least $(1-\rho^{1/3})n^k$ choices of $I'\in \mathcal{I}_{I, k}$, which satisfy $\lambda(I'\cap H) > (1-\rho^{1/3})(nm)^{-k}$. (We say that such $I, I'$ are $\rho$-good.)
    
    Proving this claim is essentially a pigeonhole principle. Indeed, otherwise for at least $\rho^{1/3}m^k$ choices of $I\in\mathcal{I}_k$, we have $\lambda(I'\cap H^c)\geq \rho^{1/3}(nm)^{-k}$ for at least $\rho^{1/3}n^k$ choices of $I'\in \mathcal{I}_{I, k}$,
    altogether implying that 
    $\lambda(H^c)=\sum_{I\in \mathcal{I}_k}\sum_{I'\in\mathcal{I}_{I, k}}\lambda(I'\cap H^c) \geq \rho,$ 
    contradicting $\lambda(H)>1-\rho$.

    Fix now a $\rho$-good $I \in \mathcal{I}_k$, and let $\mathcal{I}'_{I, k}$ denote the set of $\rho$-good $I'$ in $\mathcal{I}_{I, k}$.
    Recalling Lemma \ref{lemma:measure_multiplying}, all intervals in $\mathcal{I}'_{I, k}$ are mapped to $I$ in a measure-multiplying way by $\phi_{n, m}$, that is $\lambda(V)=n^k\lambda(\phi_{n, m}^{-1}(V)\cap I')$ for any 
    $V\subseteq I$ measurable. 
    Using the notation $B_{I'} = \phi_{n, m}(I'\cap H)$, this yields that
    $$\lambda(B_{I'}) \geq n^k\lambda(I'\cap H) > (1-\rho^{1/3})m^{-k},$$
    thus
    \begin{equation} \label{eq:pigeonhole_prep}
    \sum_{I'\in\mathcal{I}'_{I, k}}\lambda(B_{I'})\geq (1-\rho^{1/3})^2 m^{-k}n^k .
    \end{equation}
Since $H$ is measurable and an interval $I'$ is also measurable it is clear that $B_{I'}$ is also Lebesgue measurable. 
Therefore the set $$H'= \Big \{r :  \sum_{I'\in\mathcal{I}'_{I, k}}\chi_{B_{I'}}(r)\geq (1-\rho^{1/10})n^k \Big \}$$
is also measurable. Observe that $H'$ equals the set of those $r$  which are covered by $B_{I'}$ for at least
    $(1-\rho^{1/10})n^k$ elements of $\mathcal{I}'_{I, k}$. 

    As all the sets $B_{I'}$ are subsets of the interval $I$ of measure
    $m^{-k}$ next we show that
\begin{equation}\label{*lllhvest} 
 \lll(H')\geq(1-\rho^{1/10})m^{-k}.
\end{equation}

 Indeed, otherwise points of a set of measure $\rho^{1/10} m^{-k}$ are covered less than $(1-\rho^{1/10})n^k$ times, while the rest, of measure $(1- \rho^{1/10})m^{-k}$  is covered at most $n^k$ times, 
    implying that
\begin{equation}\label{*BIveq} 
   \sum_{I'\in\mathcal{I}'_{I, k}}\lambda(B_{I'}) \leq \rho^{1/10} m^{-k} (1-\rho^{1/10})n^k+ (1- \rho^{1/10})m^{-k}n^k 
\end{equation}
   $$ =(1-\rho^{2/10})m^{-k}n^{k}<
 (1-\rho^{1/3})^2 m^{-k}n^k,$$
where the last inequality holds  for small values of $\rho>0$, and at the beginning we were selecting a $\rho=\rho_{l}$ such that it holds, observe that this last inequality is not depending on $k$.   This contradicts \eqref{eq:pigeonhole_prep}.

Denote by $A$ the set of those $a$ for which 
the hyperplane $W_{a}$ is mapped by $\phi_{n,m}\circ h =
    g_{n, m, h}$ onto an element $r\in H'$.
 By our assumptions $A$ intersects at least $(1-\rho^{1/10})n^k$ distinct intervals of length $(nm)^{-k}$, and any $W_a\in \mathcal{W}$
    is such that 
    $$N_k(W_a\cap F_J) \geq (nm)^{k(s-1-\rho)}.$$
    
Recall that $h$ is non-constant and affine. Denote by $M_h$ its Lipschitz constant, which also equals the length of its gradient. Select different $a,a'\in A$ and $x\in W_{a}$, $x'\in W_{a'}$ such that $|x-x'|$ equals the distance between the parallel hyperplanes $W_{a}$ and $W_{a'}$.
Then $a=h(x)$, $a'=h(x')$ and
\begin{equation}\label{*wadis} 
  \frac{1}{M_{h}}|a-a'|=
   \frac{1}{M_{h}}|h(x)-h(x')| \leq |x-x'|=\dist(W_{a},W_{a'}).
\end{equation}
Therefore, recalling that $p$ is the dimension of our space, if 
\begin{equation}\label{*dist} 
 |a-a'|\geq M_{h} \frac{1}{(nm)^{k}}\sqrt p 
\end{equation}
then $W_{a}$ and $W_{a'}$ cannot cross the same level $k$ cube, since the distance between them exceeds the diameter of such a cube.

We have seen that $A$ intersects for small $\rho$ at least $(1-\rho^{1/10})n^k$ distinct intervals of length $1/(nm)^{k}$. Hence it contains a finite set ${\widetilde{A}}$ such that
\begin{equation}\label{*carda} 
 \# {\widetilde{A}} 
 \geq (1-\rho^{1/10})n^k  \frac{1}{2 M_{h}\sqrt p} 
 \end{equation} 
 and for any two different $a,a'\in {\widetilde{A}}$ the estimate \eqref{*dist} is satisfied. The factor $2$ in the denominator of \eqref{*carda} is taking care of the fact that neighboring intervals can have very close points, so for proper separation of points we need to estimate the number of intervals between the two intervals containing $a$ and $a'$.

    Thus
    \begin{equation} \label{eq:no_of_intersected_cubes1}
    N_k(g_{n, m, h}^{-1}(r)\cap F_J)\geq\frac{(1-\rho^{1/10})n^k \cdot (nm)^{k(s-1-\rho)}}{2M_h\sqrt p}\geq \frac{n^k \cdot (nm)^{k(s-1-\rho)}}{4M_h\sqrt p},
    \end{equation}
    where the second bound holds 
    since at the beginning we were selecting a $\rho=\rho_{l}$ such that it holds.
    
    Recalling that $\alpha_{n, m} = \frac{\log m}{\log n + \log m}$, we have
    \begin{equation} \label{eq:nm_alpha_identity}
        n=(nm)^{1-\alpha_{n, m}}.
    \end{equation}
    Plugging \eqref{eq:nm_alpha_identity} into \eqref{eq:no_of_intersected_cubes1}
    \begin{equation}\label{eq:no_of_intersected_cubes2}
        N_k(g_{n, m, h}^{-1}(r)\cap F_J)\geq\frac{1}{4M_h\sqrt p} (nm)^{k(s-\alpha_{n, m}-\rho)}.
    \end{equation}

    As $\lambda(H')\geq (1-\rho^{1/10})m^{-k}$, we have \eqref{eq:no_of_intersected_cubes2} for a set $R_\rho\subset J=[0, 1]$ of $r$s with 
    $\lambda(R_\rho)\geq 1-\rho^{1/10}$.

Until now $l$ was fixed. Now we recall that 
$\rho=\rho_l$ and $\rho_l\to 0+$ and $k_l\to\infty$ as $l\to\oo$. 
We find that almost every $r\in [0, 1]$ belongs to $
\cap_{L=1}^{\oo} \cup_{l\geq L}R_{\rho_l}$.
    Consequently,  for almost every $r \in J=[0, 1]$ there are infinitely many $\rho_{l}\leq 1/l$ and  $k_l\geq l$ for which
    $$N_{k_l}(g_{n, m, h}^{-1}(r)\cap F_J) \geq \frac{1}{4M_h \sqrt p} (nm)^{k_{l}(s-\alpha_{n, m}-\rho_l)}.$$
 Therefore, almost every level set has upper box dimension at least $s-\alpha_{n, m}$. 
\end{proof}

The following lemma will help us to show that certain sets are measurable:

\begin{lemma} \label{lemma:measurability1}
    Let $A, B, F\subseteq \mathbb{R}^p$ be bounded Borel sets and let $f:A\cup B\to \mathbb{R}$ be continuous.
    Then $\{u:\lambda(f(A\cap (F+u)))\setminus (B\cap (F+u))>0\}\subseteq \mathbb{R}^p$ is Lebesgue measurable.
\end{lemma}

\begin{proof}
    For a Borel set $H$, let $E_H = \{(u, y): y\in f(H\cap(F+u))\}\subseteq \mathbb{R}^p\times \mathbb{R}$. This is analytic, 
    as the projection of the Borel set 
    $\{(u, y, x): x\in H, x-u\in F, f(x)=y\}$ to the first two coordinates. Let $H^u=\{y:(u, y)\in H\}\subseteq \mathbb{R}$. 
    The function $g(u)=\lambda(E_{A\cup B}^{u})$ is upper semianalytic  \cite[Exercise~29.28]{Kechris1995}, i.e.,
    the $g$-preimage of every half-line $(r, +\infty)$ is analytic and hence it is Lebesgue measurable. The same holds for $h(u)=\lambda(E_B^{u})$.
    Thus $\{u:g(u)-h(u)>0\}$ is also Lebesgue measurable, which is precisely the set in question.
\end{proof}

We need a further technical lemma:

\begin{lemma} \label{lemma:values_admitted_in_simplices}
    Let $F\subseteq \mathbb{R}^p$ be compact and $\mathcal{S}$ be a finite family of simplices covering $F$ such that $\inte (S) \cap F\neq \emptyset$ for all $S\in \mathcal{S}$.
    Assume moreover that $f:\bigcup_{S\in\mathcal{S}} S \to \mathbb{R}$ is a piecewise affine function, affine on members of $\mathcal{S}$. 
    Then for almost every $u\in\mathbb{R}^p$, we have
    \begin{equation} \label{eq:values_admitted_in_interior}
    \bigcup_{S\in\mathcal{S}} f(S\cap (F+u)) = \bigcup_{S\in\mathcal{S}} f\left(\inte S\cap (F+u)\right) \text{ a. e.}
    \end{equation}
    i.e., for almost every $u$ almost every value admitted on $F+u$ by $f$ is admitted in the interior of a simplex.
\end{lemma}

\begin{proof}
    By Lemma \ref{lemma:measurability1}, the set of $u$s for which the identity of \eqref{eq:values_admitted_in_interior} fails is measurable.
    Proceeding towards a contradiction, suppose that a measurable subset $U$ of these $u$s is of positive Lebesgue measure.
    For such $u$, there is some $S_u\in \mathcal{S}$ for which a positive measure of values is admitted by $f$ in 
    $\partial S \cap (F+u)$ but not in $\inte S\cap (F+u)$. 
    Call such values $f$ \emph{degenerate values} for $F+u$. (As in this proof $f$ is fixed, we will simply speak of degenerate values below.) 
    We will pass to a subset of $U$ in a number of steps to achieve more convenient properties so that $\lambda(U)>0$ is preserved. 
    First, as $\mathcal{S}$ is finite, $S_u$ might be made independent of the choice of $u\in U$. (The set of those $u$s in $U$ for which
    for which the failure of the identity of \eqref{eq:values_admitted_in_interior} is witnessed by a given $S$ is again measurable by \Cref{lemma:measurability1}.)
    We denote this particular simplex by $S$. Note that
    $\partial S$ is the union of the relative interiors of all of its faces (including ones with non-maximal dimension),
    $$\partial S = \bigcup_{L\text{ is a face of }S} \mathrm{rel}\ \inte L.$$
    As there are finitely many such faces, we can assume that for each $u\in U$, degenerate values of a set of positive measure are admitted
    in the relative interior of a fixed, not necessarily $p-1$ dimensional, but at least 1-dimensional face $L$ of $S$, measurability
    is still preserved due to \Cref{lemma:measurability1}.
    Exhausting this relative interior with a countable subfamily of $(L_{r})_{r>0}$ defined by
    $$L_{r} = \{x\in L: \mathrm{dist}(x, \partial_{\mathrm{rel}}(L))>r\},$$
    we can assume that for each $u\in U$, degenerate values of a set of positive measure are admitted in some $L_r$.

    Fix such $r$ and denote the set of degenerate values admitted in $L_r$ by $R_u\subseteq f(L_r)$.
    Each $R_u$ is a positive measure set in $f(L_r)$, thus fixing $\varepsilon>0$, by Lebesgue's density theorem, 
    we can find a non-degenerate interval $[a_u, b_u]$ with rational endpoints such that $\lambda(R_u\cap [a, b]) > (1-\varepsilon)(b-a)$.
    As there are countably many such intervals, by further reducing $U$ we can even assume that for a 
    fixed non-degenerate interval $[a, b] \subseteq f(L_r)$, we have for every $u\in U$ that
    \begin{equation} \label{eq:in_[a,b]_mostly_boundary_values}
    \lambda(R_u\cap [a, b]) > (1-\varepsilon)(b-a).
    \end{equation}

    Pick a Lebesgue density point of $U$, by translating the entire configuration, we can assume that it is $0\in \mathbb{R}^p$. Then by $0\in U$ and \eqref{eq:in_[a,b]_mostly_boundary_values} we have
    \begin{equation} \label{eq:u=0_many_boundary_points}
    \lambda (f(\inte S \cap F)\cap [a,b])\leq \varepsilon(b-a).
    \end{equation}
    Denote by $\nu$ a unit vector such that for small enough $\delta>0$, we have $L_r-\delta \nu \subseteq \inte S$. 
    As $0$ is a density point of $U$, we can find $u_n \in U$ such that $u_n\to 0$ and $\frac{u_n}{\|u_n\|}\to \nu$. Picking $u=u_n$ with large $n$ to be fixed later,
    $L_r-u \subseteq \inte S$ is guaranteed. Consequently, we have the plain containments
    \begin{equation} \label{eq:containments_HFU}
    L_r\cap (F+u) = ((L_r-u)\cap F) + u \subseteq (\inte S \cap F) + u.
    \end{equation}
    Thus
    \begin{equation} \label{eq:equal_intersections_equal_sizes}
    \lambda(f(L_r\cap (F+u))\cap[a, b]) = \lambda(f(((L_r-u)\cap F)+u)\cap[a, b]).
    \end{equation}
    By $u\in U$, we know
    \begin{equation} \label{eq:boundary_large_intersection}
    \lambda(f(L_r\cap (F+u))\cap [a, b])> (1-\varepsilon)(b-a).
    \end{equation}
    On the other hand, as $f$ is affine on $S\supseteq (L_r-u)\cap F$, the $f$-image of $((L_r-u)\cap F) + u \subseteq S$ is simply the $f$-image of $(L_r-u)\cap F$ translated by a scalar $t_u$
    tending to 0 as $u\to 0$. Consequently,
    \begin{align} \label{eq:intersection_sizes_long_calc}
        \lambda(f(L_r\cap (F+u))\cap[a, b]) &\overset{\mathclap{\eqref{eq:equal_intersections_equal_sizes}}}{=}\ \lambda(f(((L_r-u)\cap F)+u)\cap[a, b]) \notag \\
        &\overset{\mathclap{\text{$t_u$ def.}}}{=}\ \lambda(f(((L_r-u)\cap F))\cap[a+t_u, b+t_u]) \notag \\
        &\overset{\mathclap{\eqref{eq:containments_HFU}}}{\leq }\ \lambda(f(\inte S \cap F)\cap [a+t_u, b+t_u]) \notag \\
        &\leq \lambda(f(\inte S \cap F)\cap [a-|t_u|, b+|t_u|]) \\
        &= \lambda(f(\inte S \cap F)\cap [a-|t_u|, a]) + \lambda(f(\inte S \cap F)\cap [a, b]) \notag \\ 
        &\quad+ \lambda(f(\inte S \cap F)\cap [b, b+|t_u|]) \notag\\
        &\overset{\mathclap{\eqref{eq:u=0_many_boundary_points}}}{\leq}\ \varepsilon(b-a)+2|t_u|. \notag
    \end{align}
    Comparing the first and last expressions of \eqref{eq:intersection_sizes_long_calc} and invoking the lower bound on the former given by \eqref{eq:boundary_large_intersection}, we arrive at
    $$(1-\varepsilon)(b-a) < \varepsilon(b-a)+2|t_u|,$$
    a contradiction for $\varepsilon<1/2$ and small enough $t_u$, guaranteed by large enough $n$.
\end{proof}

\section{Proof of the main result} \label{sec:main}

\begin{proof}[Proof of Theorem \ref{thm:gen_upper_box_precise}]


    Corollary \ref{corollary:lower_box_dim_upper_bound} and the subsequent remark yield that the quantity $s-\alpha$ provides an upper
    bound on $\overline{\dim}_B f^{-1}(r)$ for every $f\in C_1^{\alpha}(F)$ 
    and almost every $r\in f(F)$.
    Thus it remains to show that it is a lower bound as well for the generic $f$ and almost every $r\in f(F)$. In light of Corollary \ref{cor:connected_self_similar_is_locally_connected}, we know that $F$
    is locally connected, thus Corollary \ref{corollary:upper_box_generic_locally_connected} is applicable. Hence for $\varepsilon>0$, the to-be-shown existence of a dense subset $\mathcal{F}_\varepsilon$ of $C_1^{\alpha}(F)$ 
    functions such that for $f\in \mathcal{F}_\varepsilon$ and almost every $r\in f(F)$
    $$\overline{\dim}_B f^{-1}(r)\geq s-(\alpha+\varepsilon)$$
    implies the same for almost every level set of the generic function in $C_1^{\alpha}(F)$. Then intersecting for $\varepsilon\to 0$ over a countable sequence, we find that for the generic function $f$
    and almost every $r\in f(F)$ we have
    $$\overline{\dim}_B f^{-1}(r)\geq s-\alpha,$$
    which would conclude the proof. Thus it indeed suffices to establish the existence of $\mathcal{F}_\varepsilon$ for every $\varepsilon>0$ 
    with the aforementioned properties.
    First we will make a convenient dense subset $\mathcal{F}\subseteq C_1^{\alpha}(F)$. 
    (Note that we do not even use Theorem \ref{thm:upper_box_generic_locally_connected} to show
    that there is a generic value of $D_{\overline{B}*}^f(F)$, we get it directly from this argument.)

    Denote by $PWA(\mathbb{R}^p)$ the set of
    $1^{-}$-Hölder-$\alpha$, locally non-constant piecewise affine functions on $\mathbb{R}^p$, and by $PWA(F)$ the set of their
    restrictions to $F$.
    By Lemma \ref{lemma:piecewise_affine_approx}, $PWA(F)\subseteq C_1^{\alpha}(F)$ is dense. 
    Starting from $PWA(F)$, we will do a number of perturbations to arrive at the dense family $\mathcal{F}$. Some of these perturbation
    steps closely follow the proof of \cite[Theorem~3.10]{BUCZOLICH2024531}, however, as new ideas are needed at certain stages, we provide
    a fully detailed presentation here.

    \begin{enumerate}
        \item First, we pass to a subset of $PWA(F)$ containing typically oriented functions in the following strong sense: 
        we say that $f\in PWATO(F)$ (piecewise affine, typically oriented) if it is in $PWA(F)$ and once we put $h_1, h_2, \dots$ to be the countably many affine 
        functions defining it and extend them affinely to the full space $\mathbb{R}^p$, then each of them is typically oriented in each $F'\subseteq F$
        which is a similar piece of $F$ (obtained as the image of $F$ using a finite chain
        of the similitudes defining $F$). Having such functions will be convenient once we want to use Lemma \ref{lemma:approx_affine_functions}
        in small similar pieces of $F$.

        We claim that $PWATO(F)$ is still dense in $C_1^{\alpha}(F)$. To prove this, recall first
        that according to Proposition \ref{prop:wenxi}, for any $F'$ and $S\in \mathcal{S}$, $F'\subseteq S$, 
        we can find some subset $\mathcal{N}_{F'}
        $  in the Grassmannian $Gr(p, p-1)$ of zero
        Grassmannian measure, $\gamma_{p, p-1}$
        such that for any $W \notin \mathcal{N}_{F'}$, and almost every $a\in Pr_{W^{\perp}}(F')$, we have that 
        $\dim_H (F'\cap (W+a))=s-1$.

        Given $f\in PWA(F)$, define $h_1, h_2, \dots$ as above We will consider the perturbation 
        $\tilde{f}(x)=f(x) + \langle t, x\rangle$ for fixed $t\in \mathbb{R}^p$, where $\langle \cdot, \cdot \rangle$ denotes inner/scalar product. 
        For any $h=h_i$, $\tilde{h}(x)=h(x) + \langle t, x\rangle$.
        Now for any $r\in h(F')$, 
        $\tilde{h}^{-1}(r)\cap F' = (W_{t, h} + \tau) \cap F'$ for some $\tau \in Pr_{W^{\perp}}(F'\cap S)$ and $W_{t, h}\in Gr(p, p-1)$ depending on $t$ and $h$.
        We say that $t$ is $(F', h)$-bad, if $W_{t, h}\in \mathcal{N}_{F'}$, i.e., if $\tilde{h}$ is not typically oriented in $F'$. 
        We claim that the set of $(F', h)$-bad $t$s is of zero Lebesgue measure.
        This will be sufficient: using a small enough $t$ which is not $(F', h)$-bad for any $F', h=h_i$ results in $\tilde{f}\in PWATO(F)$, and
        this implies the density of $PWATO(F)$.

        We can determine $W_{t, h}$ explicitly: if $h(x)=\langle \alpha_h, x\rangle + c_h$
        for some $\alpha_h \in \mathbb{R}^p$ and $c_h\in \mathbb{R}$, then $W_{t, h}$ is the hyperplane orthogonal to $\alpha_h + t$. Thus 
        $\alpha_h + t$ is $(F', h)$-bad if and only if $\frac{\alpha_h + t}{\|\alpha_h + t\|_{2}}\in S^{n-1}$ is a unit normal vector of some
        $L\in \mathcal{N}_{F'}$. Denote the set of such unit normal vectors by $E_{\mathcal{N}_{F'}}$. 
        As $\mathcal{N}_{F'}$ has zero  $\gamma_{p, p-1}$ measure
        on $Gr(p, p-1)$, $E_{\mathcal{N}_{F'}}\subseteq \mathbb{S}^{p-1}$ has zero spherical measure.
        That is, we
        need that $\frac{\alpha_h + t}{\|\alpha_h + t\|_{2}}$ evades the null-set $E_{\mathcal{N}_{F'}}$ for almost every $t$. 
        After embedding $\mathbb{S}^{p-1}$ into $\mathbb{R}^p$, we see that this
        evasion happens if and only if $\alpha_h + t$ evades the $p$-dimensional Lebesgue null-set 
        $\bigcup_{\lambda>0}\lambda E_{\mathcal{N}_{F'}} \subseteq \mathbb{R}^{p}$. (This set is a null-set due
        to Fubini's theorem.)
        That is, the set of $(F', h)$-bad $t$s is of zero Lebesgue measure, as we claimed. 
        After taking countable union over all $F'$ and $h_1, h_2, \dots$, we find that there exist $t$s with arbitrarily small norm which are
        not $(F', h)$-bad for any choice of $F'$ and $h=h_i$, implying the density of $PWATO(F)\subseteq C_1^{\alpha}(F)$.

        We take note of the fact that if $f\in PWATO(F)$, then $\tilde{f}(x)=f(x+c)$ is also in $PWATO(F)$. In the next two steps,
        we will only use such shifting perturbations, which cannot lead out of $PWATO(F)$.

        \item Now we achieve that our simplices intersect $F$ substantially. Notably, for $f\in PWATO(F)$, if $\mathcal{S}$
        is the family of simplices used in its definition, we say that $S\in\mathcal{S}$ is {\it non-boundary} 
        for $F$, if either $\inte S \cap F \neq \emptyset$, or $S\cap F = \emptyset$. 
        We say that $f\in PWATO(F)$ belongs to $PWANB(F)$, if each $S\in\mathcal{S}$ is {\it non-boundary} 
        for $F$. We claim that $PWANB(F)$ is still a dense subset of $C_1^{\alpha}(F)$.

        To prove this claim, we show that one can select an arbitrarily small $c\in\mathbb{R}^p$, such
        that all the simplices $\tilde{\mathcal{S}}= \bigcup_{S\in\mathcal{S}}\{S-c\}$ are non-boundary for $F$.
        Then the function $\tilde{f}(x) = f(x+c)$ belongs to $PWATO(F)$.
        As $f$ is uniformly continuous, the existence of arbitrarily small such $c$s would imply the denseness of $PWANB(F)$ 
        due to the denseness of $PWATO(F)$.

        Take finite $\mathcal{S}_0\subseteq \mathcal{S}$ such that $F\subseteq \inte\left(\bigcup \mathcal{S}_0\right)$.
        Fix $S\in \mathcal{S}_0$ and denote by $A_S$ the set of $c$s for which $S-c$ is a non-boundary simplex. Then $A_S$ is clearly open. 
        Moreover, if $c\notin A_S$, then $S-c$ contains some $x\in F$ on its boundary,
        which yields the existence of some $c'\in A_S$ arbitrarily close to $c$, ($c'=c+v$ is satisfying for any $v$ which points into $S$ from $x$). 
        Thus $A_S$ is dense. However, then $A=\bigcap_{S\in \mathcal{S}_0}A_S$
        is still dense and open, and using any $c\in A$ yields $\tilde{f}\in PWANB(F)$. Thus $PWANB(F)$ is dense indeed.

        We take note of the fact that once $f\in PWANB(F)$, $\tilde{f}(x)=f(x+c)$ is also in $PWANB(F)$ for small enough $c$.

        \item Now we invoke Lemma \ref{lemma:values_admitted_in_simplices}
        to see that for a dense subset $\mathcal{F}\subseteq PWANB(F)$, for every $f\in \mathcal{F}$
        almost every value of the range of $f$ is admitted in $\inte S \cap F$ for some $S\in\mathcal{S}$. Indeed, that lemma guarantees
        a dense set of translations of the set $F$ with this property, which can be equivalently formulated to guarantee a dense set of translations 
        of $f$. In this dense set of translations, we can find arbitrarily small ones which guarantee that the shifted function is still in $PWANB(F)$.
    \end{enumerate}
    
    We use $\mathcal{F}$
    as a starting point to construct $\mathcal{F}_\varepsilon$. More precisely, we use the auxiliary 
    functions $\phi_{n, m}$ as in Lemma \ref{lemma:approx_affine_functions} to perturb $f\in\mathcal{F}$ for $n, m$ such that 
    $\alpha+\frac{\varepsilon}{2}<\alpha_{n, m}<\alpha + \varepsilon$.
    Such $n,m$ exists for arbitrarily large $m$. For some  $K_{\varepsilon, f}>0$ to be fixed later, we call such $n,m$ $(\varepsilon, f)$-permitted, if even $m>K_{\varepsilon, f}$ holds.
    For any $f\in \mathcal{F}$, put $g_{n, m, f}(x) = \phi_{n, m}(f(x))$, and let
    $$\mathcal{F}_\varepsilon = \{g_{n, m, f}: f\in\mathcal{F}, n,m \text{ are } (\varepsilon,f)\text{-permitted}\}.$$

    As $\mathcal{F}$ was dense in $C_1^{\alpha}(F)$, and $\|\phi_{n, m} - \Id_{\mathbb{R}}\|_{\infty}\leq \frac{1}{m}$, and there are $(\varepsilon, f)$-permitted pairs $n,m$ with arbitrarily large $m$,
    we readily see that $\mathcal{F}_\varepsilon$ is dense as well. 
    
    It is also simple to check that any $g_{n, m, f}\in \mathcal{F}_\varepsilon$ is 1-Hölder-$\alpha$. To this end, fix
    $x,y \in F$. As $f\in C_{1-}^{\alpha}(F)$, we can choose $c<1$ such that $f\in C_c^{\alpha}(F)$. Then the change of $g_{n, m, f}$ between $x, y$ is
    \begin{equation}\label{eq:change_of_gnm}
    \Delta_{x,y} :=|g_{n, m, f}(x) - g_{n, m, f}(y)| = |\phi_{n, m}(f(x)) - \phi_{n, m}(f(y))|.
    \end{equation}
    We give two bounds on this quantity, and it will turn out that they cover any choice of $x,y$ for large enough $K_{\varepsilon, f}$. First of all, as  
    $\|\phi_{n, m} - \Id_{\mathbb{R}}\|_{\infty}\leq \frac{1}{m}$, using $m>K_{\varepsilon, f}$ and \eqref{eq:change_of_gnm} we see that 
    \begin{equation} \label{large_diff_bound}
    \Delta_{x, y}\leq |f(x) - f(y)| + \frac{2}{K_{\varepsilon,f}} \leq c|x-y|^{\alpha} + \frac{2}{K_{\varepsilon,f}}.
    \end{equation}
    On the other hand, as $f$ is $M_f$-Lipschitz and $\phi_{n,m}$ is 1-Hölder-$\alpha_{n, m}$ in any interval of length 1,
    whenever $M_f|x-y|<1$ we have
    \begin{equation} \label{small_diff_bound}
        \Delta_{x, y}\leq (M_f|x-y|)^{\alpha_{n, m}} \leq (M_f|x-y|)^{\alpha+\frac{\varepsilon}{2}}.
    \end{equation}

    Now we only have to note that the bound in \eqref{small_diff_bound} implies $\Delta_{x, y}\leq |x-y|^\alpha$, if $|x-y|$ is small enough compared to $M_f$, while the bound in \eqref{large_diff_bound}
    implies the same if $|x-y|$ is large enough compared to $\frac{1}{K_{\varepsilon,f}}$. Thus we can fix $K_{\varepsilon,f}$ large enough so that \eqref{large_diff_bound} handles 
    any pair $x,y$ not already handled by \eqref{small_diff_bound}. This choice of $K_{\varepsilon, f}$ yields that $\mathcal{F}_\varepsilon$ is indeed a dense subset of $C_1^{\alpha}(F)$.

    Now it suffices to check the dimensions of the level sets, i.e., that they exceed $s-(\alpha+\varepsilon)$.
    To this end, observe that $f$ admits almost every value of its range in the interior of some $S$. This implies the same for $g_{n, m, f}$:
    to show this, let $A\subseteq \mathbb{R}$ be the nullset of values admitted by $f$ 
    on the boundaries of the simplices but not in their interiors.
    Recall that $\phi_{n, m}$ is Lebesgue measure preserving, that is for every measurable set $B$ we have $\lll(\phi_{n, m}^{-1}(B))=\lll(B)$.
    Let $B\subseteq \mathbb{R}$ be the set of values admitted by $g_{n,m,f}$ 
    on the boundaries of the simplices but not in their interiors. We need to show that $\lll(B)=0$.

    Proceeding towards a contradiction suppose that $\lll(B)>0$. Then $\lll(\phi_{n, m}^{-1}(B))>0$
    holds as well. This implies that there is $r\in \phi_{n, m}^{-1}(B) \sm A$. That is 
    $\phi_{n, m}(r)=g_{n,m,f}(x)=\phi_{n, m}(f(x))$ for a suitable $x\in F$, but there is no $y$ belonging to the interior of some $S$, 
    satisfying $\phi_{n, m}(f(y))=\phi_{n, m}(f(x)).$ Therefore, there is no $y$ belonging to the interior of some $S$ satisfying $f(y)=f(x)=r$, 
    but this means that $r\in A$, a contradiction.

    Thus we found that almost every value of $g_{n, m, f}$ 
    is admitted in some point in the interior of some $S$, which point in turn is
    contained by a similar image $F_1\subset F$ of $F$ fully in $S$. Thus Lemma \ref{lemma:approx_affine_functions} can be applied.

\end{proof}

\section{Concluding remarks} \label{sec:concluding}

We would like to highlight two interesting open problems:

\begin{question}
    How Theorem \ref{thm:gen_upper_box_precise} generalizes to $s<1$?
\end{question}

\begin{question} 
    Is $F$ having nice connection type equivalent to $F$ being locally connected?
\end{question}

\bibliographystyle{amsplain} 
\bibliography{sierbox} 

\end{document}